\pdfoutput=1
\RequirePackage[l2tabu, orthodox]{nag}

\documentclass[reqno]{article}
\usepackage{e-jc}

\usepackage{lmodern}
\usepackage[T1]{fontenc}
\usepackage[utf8]{inputenc}
\usepackage[english]{babel}
\usepackage{microtype} 

\usepackage{amsmath,amssymb,amsthm,mathrsfs,latexsym,mathtools,mathdots,enumerate,tikz,bm,url}
\usepackage[enableskew,vcentermath]{youngtab}
\usepackage{ytableau}

\usetikzlibrary{positioning}
\usepackage[enableskew,vcentermath]{youngtab}

\usepackage[pdftex]{hyperref}
\usepackage[capitalize]{cleveref}

\usepackage{todonotes}
\presetkeys%
    {todonotes}%
    {inline,backgroundcolor=yellow}{}

\usepackage{graphicx}

\usepackage{ifpdf}
\graphicspath{{./figures/}}
\usepackage{epstopdf}

\DeclareGraphicsExtensions{.png,.jpg,.pdf}

\newtheorem{theorem}{Theorem}
\newtheorem{proposition}[theorem]{Proposition}
\newtheorem{lemma}[theorem]{Lemma}

\newtheorem{conjecture}[theorem]{Conjecture}

\theoremstyle{definition}

\newtheorem{example}[theorem]{Example}
\newtheorem{remark}[theorem]{Remark}

\newtheorem{claim}[theorem]{Claim}

\newcommand{\defin}[1]{\emph{#1}}

\newcommand{\setC}{\mathbb{C}}

\newcommand{\xvec}{\mathbf{x}}

\newcommand{\tpi}{\tilde{\pi}}
\newcommand{\ttheta}{\tilde{\theta}}

\newcommand{\symS}{S}

\newcommand{\key}{\mathcal{K}}

\newcommand{\macdonaldE}{\mathrm{E}}
\newcommand{\macdonaldH}{{\mathrm{\tilde{H}}}}
\newcommand{\schurS}{\mathrm{s}}

\newcommand{\CoInvFree}{\mathrm{CoInvFree}}
\newcommand{\InvFree}{\mathrm{InvFree}}
\newcommand{\NAF}{\mathrm{NAF}}
\newcommand{\FIL}{\mathrm{FIL}}
\newcommand{\SSYT}{\mathrm{SSYT}}

\DeclareMathOperator{\Des}{\mathrm{Des}}

\DeclareMathOperator{\leg}{leg}
\DeclareMathOperator{\arm}{arm}
\DeclareMathOperator{\inv}{inv}
\DeclareMathOperator{\maj}{maj}
\DeclareMathOperator{\coinv}{coinv}

\DeclareMathOperator{\charge}{charge}
\DeclareMathOperator{\sort}{sort}
\DeclareMathOperator{\cocharge}{cocharge}
\DeclareMathOperator{\reverse}{rev}
\DeclareMathOperator{\cw}{cw}
\DeclareMathOperator{\ccw}{ccw}

\title{A major-index preserving map on fillings}
\author{Per Alexandersson and Mehtaab Sawhney}

\begin{document}

\maketitle

\begin{abstract}
We generalize a map by S. Mason regarding two combinatorial models for key polynomials,
in a way that accounts for the major index.

We also define similar variants of this map, that regards alternative 
models for the modified Macdonald polynomials at $t=0$, 
thus partially answer a question by J. Haglund.

These maps imply certain uniqueness property regarding inversion-- and coinversion-free fillings,
which allows us to generalize the 
notion of charge to a non-symmetric setting, thus answering a question by A. Lascoux. 
The analogous question in the symmetric setting proves a conjecture by K. Nelson.
\end{abstract}



\section{Introduction}

The area of Macdonald polynomials and related combinatorics has been very active in the last $10$ years.
In 2006, James Haglund, Mark Haiman and Nick Loehr described a combinatorial formula 
for the non-symmetric Macdonald polynomials, \cite{HaglundNonSymmetricMacdonald2008}.
This combinatorial model specializes to a model for Demazure characters, 
or \emph{key polynomials} and \emph{Demazure atoms},
which is studied in \emph{e.g.} \cite{Pun2016Thesis, Mason2009}.

This model includes a \emph{basement}, a certain parameter $\sigma \in \symS_n$, 
which can be modeled combinatorially, or via Demazure--Lusztig 
operators, see \cite{Ferreira2011,Alexandersson15gbMacdonald}.
A generalization to other types using the Ram--Yip combinatorial 
model \cite{RamYip2011} can be found in \cite{Feigin2016}.
The fillings we consider in this paper are the combinatorial objects that 
generate the specialization $\macdonaldE^\sigma_\alpha(\xvec;q,0)$ 
of \emph{permuted-basement non-symmetric Macdonald polynomials} and the closely related  $\macdonaldH^\sigma_\alpha(\xvec;q,0)$ which are a specialization of
\emph{permuted-basement modified Macdonald polynomials}.

Note that the non-symmetric Macdonald polynomial $\macdonaldE_\alpha(\xvec;q,t)$
specialize to the \emph{key polynomial} (or Demazure character) $\key_\alpha(\xvec)$ at $q=t=0$,
so the specialization $\macdonaldE_\alpha(\xvec;q,0)$ can be considered as a
$q$-deformation of key polynomials.
These can be seen as a non-symmetric extension of 
the modified Hall--Littlewood polynomials $\macdonaldH_\lambda(\xvec;q)$ in the following sense:
we have that $\omega \macdonaldE_\lambda(\xvec;q,0) = q^{\ast(\lambda)} \macdonaldH_\lambda(x_1,\dotsc,x_n;q^{-1})$ whenever $\lambda$
is a partition of length $n$, and $\ast(\lambda)$ is an appropriate integer\footnote{The maximum number of inversion triples in the diagram of shape $\lambda$.}.
This identity follows from properties of LLT polynomials, see \emph{e.g.,} \cite{Haglund2005Macdonald}.

\subsection{Overview of results}

We construct maps between certain fillings that give the evaluation of non-symmetric Macdonald polynomials at
$t=0$. These maps preserve the major index statistic.
The existence of such maps are implied by relations given by Demazure operators,
but have not been constructed explicitly. 
We construct such maps with particularly nice properties not implied by the operators themselves ---
in particular, we show that the maps can be made to preserve column sets.
These maps and their properties allow us to solve several problems in this area:

\begin{itemize}

\item In \cite{Mason2009}, two models for key polynomials are given, with a column-set preserving 
bijection showing that these are equal.
We generalize this map to incorporate a $q$-parameter, corresponding to major index.
It is worth noting that even when $q=0$, the bijection between the sets we consider is nontrivial ---
this case was treated in \cite{Kurland2016} and the map given in the present paper specializes to the one in \cite{Kurland2016} when dealing with $q=0$.

\item We explicitly construct a biword for coinversion-free fillings
and extend the biword given for modified Macdonald polynomials in \cite{Haglund2005Macdonald}.
In the first case, the charge of the biword is shown to be equal to the major index of the filling, 
while in the second case the cocharge of the biword is equal to the major index of the filling. 
The extension of the second biword beyond the partition case
proves a generalization of a conjecture given in \cite{Nelsen2005}.

\item We demonstrate a bijective proof that 
 $\macdonaldH^\sigma_{\sigma \lambda}(\xvec;q,0) = \macdonaldH_{\lambda}(\xvec;q,0)$,
 for any fixed basement $\sigma \in \symS_n$.
 In particular, we show that the bijection between the corresponding fillings can be taken to be \emph{column-set-preserving},
 a property that uniquely defines this bijection.
\end{itemize}

Our proof method amounts to first constructing the maps for fillings with two rows, and then proving that
these maps are compatible with a larger filling.
The second part of this proof is partially done through computer verification, 
due to the large (but finite) number of cases that needs to be considered.
\medskip 

The paper is structured as follows:
In \cref{sec:preliminaries}, we introduce the necessary terminology regarding the 
combinatorial model for non-symmetric Macdonald polynomials.
In \cref{sec:main} and \cref{sec:main2}, we define and prove properties of the maps. 
Finally, in \cref{sec:applications}, we examine the various applications and consequences of the maps.

\section{Non-symmetric Macdonald polynomials and fillings}\label{sec:preliminaries}

In this section, we review the necessary terminology regarding the combinatorial model for non-symmetric Macdonald polynomials,
and modified Macdonald polynomials. We use the same notation as in 
\cite{Alexandersson15gbMacdonald}, which differs slightly from the one used in \cite{qtCatalanBook,HaglundNonSymmetricMacdonald2008}.
In particular, we use \emph{English notation}, and not the ``skyline'' way of presenting fillings.

\medskip

An \emph{augmented diagram} of shape $\alpha$ is a Young diagram where the length of row $i$ from the top is given by $1+\alpha_i$.
The leftmost column is considered special and is referred to as the \emph{basement}.
An \emph{augmented filling} is an assignment of natural numbers to the boxes in the diagram.
We specify the entries in the basement by listing them from top to bottom --- in most places,
the basement is a permutation expressed in the one-line notation.
The \emph{weight} of a filling is the multiset of entries that are not part of the basement.
We let $\xvec^F \coloneqq \prod_{u \in F} x_u$, where $u$ ranges over all non-basement entries in $F$.

\begin{example}\label{ex:nonAttackingFilling}
Below is an augmented filling with shape $(2,4,0,3,2)$, basement given by $(4,5,1,3,2)$ and 
$\xvec^F = x_1^2 x_2^2 x_3^3x_4 x_5^2$.
\[
\begin{ytableau}
\mathbf 4 & 1 & 2 \\
\mathbf 5 & 5 & 4 & 1 & 3\\
\mathbf 1  \\
\mathbf 3 & 3 & 3 & 5 \\
\mathbf 2 & 2
\end{ytableau}
\]
\end{example}

Let $F$ be an augmented filling. Two boxes $a$, $b$, are said to be \defin{attacking}
if $F(a)=F(b)$ and the boxes are either in the same column,
or they are in adjacent columns with the rightmost box in a row strictly below the other box.
A filling is \defin{non-attacking} if there are no attacking pairs of boxes. 
The filling in \cref{ex:nonAttackingFilling} is non-attacking.

\subsection{Inversions, coinversions and descents}

A \defin{triple of type $A$} is an arrangement of boxes, $a$, $b$, $c$,
located such that $a$ is immediately to the left of $b$, and $c$ is somewhere below $b$,
and the row containing $a$ and $b$ is at least as long as the row containing $c$.
In a similar fashion, a \defin{triple of type $B$} is an arrangement of boxes, $a$, $b$, $c$,
located such that $a$ is immediately to the left of $b$, and $c$ is somewhere above $a$,
and the row containing $a$ and $b$ is \emph{strictly} longer than the row containing $c$.

A type $A$ triple is an \defin{inversion triple} if the entries ordered increasingly,
form a \emph{counter-clockwise} orientation. Similarly, a type $B$ triple is an inversion triple
if the entries ordered increasingly form a \emph{clockwise} orientation.
In the case of equal entries, the one with largest subscript in \cref{eq:invTriplets}
is considered to be largest.
\begin{equation}\label{eq:invTriplets}
\text{Type $A$:}\quad
\ytableausetup{centertableaux,boxsize=1.2em}
\begin{ytableau}
 a_3 & b_1 \\
 \none  & \none[\scriptstyle\vdots] \\
\none & c_2 \\
\end{ytableau}
\qquad
\text{Type $B$:}\quad
\ytableausetup{centertableaux,boxsize=1.2em}
\begin{ytableau}
c_2 & \none \\
\none[\scriptstyle\vdots]  & \none \\
a_3 & b_1 \\
\end{ytableau}
\end{equation}
A triple which is not an inversion triple is called a \emph{coinversion triple}.

\medskip 

Let $F$ be an augmented filling and suppose $b$ is a non-basement box, 
and $a$ is the box immediately to the left of $b$.
We say that $b$ is a \defin{descent} of $F$ if $F(a) < F(b)$.
The set of descents of $F$ is denoted $\Des(F)$.

The \defin{leg} of a box $u$, $\leg(u)$, in an augmented diagram is the number of boxes to the right of $u$.
The \defin{arm}, $\arm(u)$, of $u = (r,c)$ in an augmented diagram $\alpha$ is the total cardinality of
the sets
\begin{align*}
\{ (r', c) \in \alpha : r < r' \text{ and } \alpha_{r'} \leq \alpha_r \} \text{ and } \\
\{ (r', c-1) \in \alpha : r' < r \text{ and } \alpha_{r'} < \alpha_r \}.
\end{align*}
Given an augmented filling $F$, the \emph{major index}, 
$\maj(F)$, is defined as 
\[
\maj(F) \coloneqq \sum_{u \in \Des(F)} \leg(u)+1,
\]
and the number of inversions, $\inv(F)$ is the number of inversion triples in $F$.
Similarly, $\coinv(F)$ is the number of coinversion triples in $F$.

\medskip

Let $\sigma \in \symS_n$ and let $\alpha$ be a composition with $n$ parts,
and let $\NAF(\alpha,\sigma)$ denote the set of non-attacking fillings of the augmented diagram of shape $\alpha$
and basement $\sigma$, with entries in $1\dotsc,n$.
The \defin{non-symmetric permuted basement Macdonald polynomial} 
$\macdonaldE^\sigma_\alpha(\xvec;q,t)$ is defined as
\begin{equation}\label{eq:nonSymmetricMacdonaldBasement}
\macdonaldE^\sigma_\alpha(\xvec; q,t) = \sum_{ F \in \NAF(\alpha,\sigma)} \xvec^F q^{\maj F} t^{\coinv F} \!\!\!
\prod_{ \substack{ u \in F \\ F(u_-)\neq F(u) }} \!\!\! \frac{1-t}{1-q^{1+\leg u} t^{1+\arm u}},
\end{equation}
where $u_-$ denotes the box to the left of $u$,
and we consider $F(u_-)$ not to be equal to $F(u)$ if $u$ is a box in the basement.

The ordinary non-symmetric Macdonald polynomial $\macdonaldE_\alpha(\xvec;q,t)$ considered in \cite{HaglundNonSymmetricMacdonald2008}
is recovered when taking $\sigma = w_0 = (n,\dotsc,2,1)$, that is, the unique longest permutation in $\symS_n$. 
The basement $w_0$ is often referred to as \emph{the key basement} --- the reason will be evident further down.

\begin{example}
 The set $\NAF(\alpha,\sigma)$ for $\alpha = (1,0,2,2)$, $\sigma=(2,1,3,4)$
 consists of the following augmented fillings:
\begin{align*}
\substack{\young(21,1,332,444)\\ \coinv: 0\\ \maj: 0} \quad
\substack{\young(21,1,333,442)\\ \coinv: 0\\ \maj: 0} \quad
\substack{\young(21,1,333,444)\\ \coinv: 0\\ \maj: 0} \quad
\substack{\young(21,1,334,442)\\ \coinv: 2\\ \maj: 1}\\
\substack{\young(22,1,331,444)\\ \coinv: 1\\ \maj: 0} \quad 
\substack{\young(22,1,333,441)\\ \coinv: 1\\ \maj: 0} \quad 
\substack{\young(22,1,333,444)\\ \coinv: 0\\ \maj: 0} \quad 
\substack{\young(22,1,334,441)\\ \coinv: 3\\ \maj: 1}
\end{align*}
\end{example}
\bigskip

Furthermore the \emph{modified Macdonald polynomials} are a class of symmetric functions, 
defined\footnote{This is the same definition as in \cite{Haglund2005Macdonald}, 
by using the fact that $\macdonaldH_\alpha(\xvec; q,t) = \macdonaldH_{\alpha'}(\xvec; t,q)$.
} via
\begin{equation}\label{eq:symmetricMacdonaldDef}
\macdonaldH_\alpha(\xvec; q,t) = \sum_{ F \in \FIL(\alpha,w_0)} \xvec^F q^{\maj F} t^{\inv F}
\end{equation}
where $\FIL(\alpha,w_0)$ is the set of fillings of \emph{partition shape} $\alpha$ 
with no restriction whatsoever.
The basement $w_0$ is now a \emph{big basement} --- a decreasing sequence of ``infinities'',
that is, $w_0 = (\infty_n,\infty_{n-1},\dotsc,\infty_1)$,
where we consider $\infty_i < \infty_j $ if $i<j$,
and $\infty_i>k$ for all natural numbers $i$, $k$ in the definition of major index
and inversion triples, see \cite{Haglund2005Macdonald} for details.

It is possible to generalize the modified Macdonald polynomials 
to \emph{permuted basement modified Macdonald Polynomials} by considering a 
big basement $\sigma$ as the sequence $(\infty_{\sigma_1},\infty_{\sigma_{2}},\dotsc,\infty_{\sigma_n})$. Then define
\begin{equation}\label{eq:symmetricMacdonaldDef2}
\macdonaldH^\sigma_\alpha(\xvec; q,t) = \sum_{ F \in \FIL(\alpha,\sigma)} \xvec^F q^{\maj F} t^{\inv F}
\end{equation}
where $\FIL(\alpha,\sigma)$ is the set of fillings of \emph{composition shape} $\alpha$ 
with no restriction whatsoever. 
However, one can show that these polynomials (up to a multiple of a power of $t$) 
only depend on the parts of $\alpha$, see \emph{e.g.} \cref{eq:generalMacdonaldHIdentity} further down.
This fact is not easy to see from the definition here.

\subsection{Specializations of Macdonald polynomials}

The main topic of this paper is the 
specialization of \eqref{eq:nonSymmetricMacdonaldBasement} and \eqref{eq:symmetricMacdonaldDef2}
at $t=0$. In the first case we have the simplified expression
\begin{equation}\label{eq:qAtoms}
\macdonaldE^\sigma_\alpha(\xvec; q,0) = \sum_{ F \in \CoInvFree(\alpha,\sigma)} \xvec^F q^{\maj F}
\end{equation}
where $\CoInvFree(\alpha,\sigma)$ is the set of coinversion-free fillings of shape $\alpha$ and basement $\sigma$.
Note that a coinversion-free filling is also automatically non-attacking.
We define the \emph{key polynomial} $\key_\alpha(\xvec)$ as the 
specialization $\macdonaldE^{w_0}_\alpha(\xvec; 0,0)$ --- note the use of the key basement $w_0$.

\emph{Caution:} our notation differs slightly from other literature: $\kappa_{\alpha} = \key_{\reverse(\alpha)}$,
where $\kappa_{\alpha}$ is the notation for key polynomials in \emph{e.g.} \cite{ReinerShimozono1995}.
Our notation fulfills the relation $\key_{\lambda}(x_1,\dotsc,x_n) = \schurS_{\lambda}(x_1,\dotsc,x_n)$
whenever $\lambda$ is a partition with $n$ parts.

\bigskip 

Similarly, we have the simplified expression
\begin{equation}\label{eq:basementPermuted}
\macdonaldH^\sigma_\alpha(\xvec; q,0) = \sum_{ F \in \InvFree(\alpha,\sigma)} \xvec^F q^{\maj F}
\end{equation}
where $\InvFree(\alpha,\sigma)$ is the set of inversion-free fillings of shape $\alpha$ and basement $\sigma$. 
The standard \emph{modified Hall--Littlewood polynomials}, $\macdonaldH_\alpha(\xvec; q,0)$ are recovered 
by letting $\sigma=w_0$ and $\alpha$ be a partition.

\begin{figure}[!ht]
\includegraphics[width=0.9\textwidth]{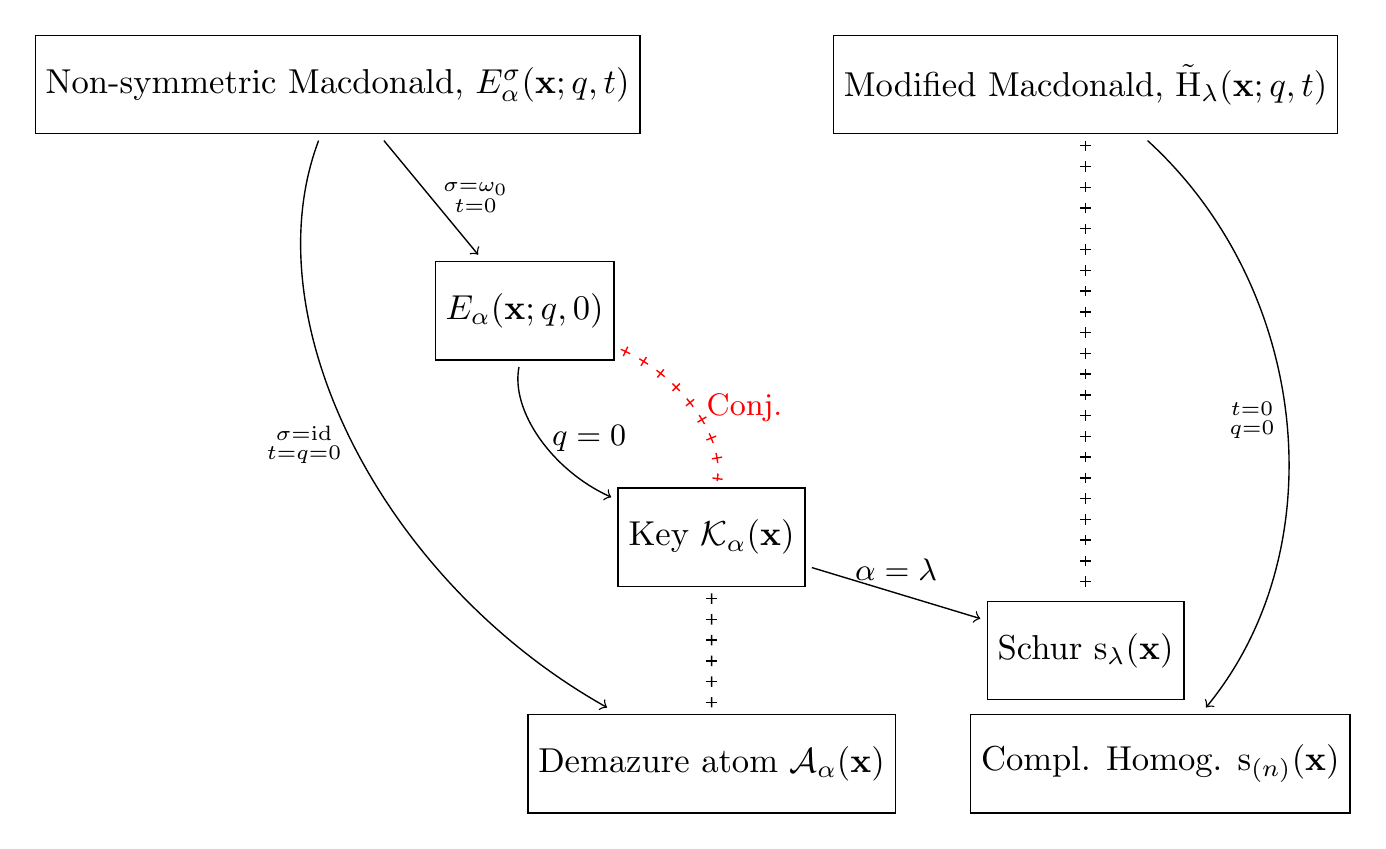}
\caption{
Here is an overview of the polynomials.
Arrows down indicate the relation \emph{specializes to},  and plus-arrows down indicate \emph{expands positively into}.
Note that one relation is \cref{conj:positiveKeyExp}, which we discuss further down.
Here, $\lambda$ is a partition, and $\alpha$ is a composition.
}\label{fig:polynomialsOverview}
\end{figure}

\subsection{Brief background on Demazure operators}\label{subsec:operators}

We only use the following operators briefly in \cref{prop:bijExistence} below,
and this subsection is only to give some context.
\medskip 

The non-symmetric Macdonald polynomials and the more general permuted basement Macdonald polynomials 
can also be defined using Demazure--Lusztig operators,
see \emph{e.g.}, \cite{HaglundNonSymmetricMacdonald2008,Alexandersson15gbMacdonald} for details.
\medskip 

Let $s_i$ act by simple transposition on the indices of the $x_i$, and define
\[
 \partial_i = \frac{1-s_i}{x_i-x_{i+1}}, \quad \pi_i = \partial_i x_i, \quad \theta_i = \pi_i - 1.
\]
It is straightforward to see that $\partial_i$, $\pi_i$ and $\theta_i$ are operators on $\setC[\xvec]$.
The operators $\pi_i$ and $\theta_i$ are used to define the key polynomials and Demazure atoms, 
see \emph{e.g.} \cite{Lascoux1990Keys,Mason2009}.
We define the following $t$-deformations of the above operators.
\begin{align}
\tpi_i(f) = (1-t)\pi_i(f) + t s_i(f)  \qquad \ttheta_i(f) = (1-t)\theta_i(f) + t s_i(f).
\end{align}
The $\ttheta_i$ are the \emph{Demazure--Lusztig operators}, and generators of the affine Hecke algebra.
The $\tpi_i$ and $\ttheta_i$ both satisfy the braid relations, and $\tpi_i \ttheta_i =  \ttheta_i\tpi_i = t$.
These operators act on the basement of permuted basement Macdonald polynomials, as well as the indexing composition.

\section{A column-set preserving map}\label{sec:main}

By using the operators in \cref{subsec:operators} together with 
the properties proved in \cite{Alexandersson15gbMacdonald},
it is possible to show the existence of a weight preserving and 
major index preserving bijection between two sets of coinversion-free fillings:
\begin{proposition}\label{prop:bijExistence}
Let $\CoInvFree(\alpha,\sigma)$ be the set of coinversion-free fillings with shape $\alpha$ and basement $\sigma$.
Suppose $\sigma_i = \sigma_{i+1}+1$ and $\alpha_i < \alpha_{i+1}$. Then there exists a bijection
\begin{align}\label{eq:bijExistence}
 \phi : \CoInvFree(\alpha,\sigma) \longleftrightarrow \CoInvFree(s_i \alpha,\sigma) \sqcup \CoInvFree(s_i \alpha, s_i\sigma)
\end{align}
with the property that $\phi$ preserves major index.
\end{proposition}
\begin{proof}
Using relations in \cite{Alexandersson15gbMacdonald}, we have
\begin{align}
 \macdonaldE^\sigma_{\alpha}(\xvec;q,0) &= \pi_i \macdonaldE^\sigma_{s_i \alpha}(\xvec;q,0) \label{eq:shapeperm}\\
 &= (1+\theta_i) \macdonaldE^\sigma_{s_i \alpha}(\xvec;q,0) \notag \\
 &= \macdonaldE^\sigma_{s_i \alpha}(\xvec;q,0) +  \macdonaldE^{s_i\sigma}_{s_i \alpha}(\xvec;q,0) \notag,
\end{align}
where $s_i$ acts via simple transpositions on the parts of $\alpha$.

Since $\macdonaldE^\sigma_{\alpha}(\xvec;q,0)$ is the weighted sum over the elements in 
$\CoInvFree(\alpha,\sigma)$, and $\macdonaldE^\sigma_{s_i \alpha}(\xvec;q,0)$ 
and $\macdonaldE^{s_i\sigma}_{s_i \alpha}(\xvec;q,0)$
are the weighted sums over $\CoInvFree(s_1 \alpha,\sigma)$ and $\CoInvFree(s_1 \alpha,\sigma)$ respectively.
It follows that a weight-preserving bijection must exist.
\end{proof}

\bigskip 
Given a filling, its \emph{column-sets} is simply the list of (multi)sets of entries in each column.
For example, the filling
\[
 \young(22,1,331,444) \text{ has column sets } (\{1,2,3,4\}, \{2,3,4\}, \{1,4\}).
\]
The purpose of this paper is to explicitly construct a bijection $\phi$
with the additional property that it is \emph{column-set preserving},
that is, $F$ and $\phi(F)$ have the same column sets.
Note that it is not clear a priori that we can impose such a strong condition on $\phi$.
\medskip 

In order construct such a $\phi$ --- which turns out to be unique ---
we first prove the following statement:
\begin{proposition}\label{prop:colSetPresInj}
Let $\CoInvFree(\alpha,\sigma)$ be the set of coinversion-free fillings with shape $\alpha$ and basement $\sigma$.
Suppose $\sigma_i >\sigma_{i+1}$ and $\alpha_i > \alpha_{i+1}$. Then there is an \emph{injection}
\[
 \phi : \CoInvFree(\alpha,\sigma) \hookrightarrow \CoInvFree(s_i \alpha,\sigma) \sqcup \CoInvFree(s_i \alpha, s_i\sigma)
\]
with the property that $\phi$ preserves column sets and major index. 

Furthermore, $\phi$ is a bijection whenever $\sigma_i=\sigma_{i+1}+1$.
\end{proposition}
The first part of this proposition together with \cref{prop:bijExistence} implies that $\phi$
is a bijection whenever $\sigma_i=\sigma_{i+1}+1$.

The proof of \cref{prop:colSetPresInj} is broken into two major parts: 
we first construct $\phi$ for fillings with two rows in \cref{lem:tworowCase},
and then show that the result is compatible with the remainder of the filling
in \cref{lem:fillingcompatible}.

\begin{lemma}\label{lem:tworowCase}
Let $\CoInvFree(\alpha,\sigma)$ be the set of coinversion-free two-row
fillings with shape $\alpha=(\alpha_1,\alpha_2)$ and basement $\sigma=(\sigma_1, \sigma_2)$
with $\sigma_1 >\sigma_{2}$ and $\alpha_1 > \alpha_{2}$. 
Then there is a column-set and major-index preserving injection $\phi$
\[
 \phi : \CoInvFree(\alpha,\sigma) \hookrightarrow \CoInvFree(s_1 \alpha,\sigma) \sqcup \CoInvFree(s_1 \alpha, s_1\sigma).
\]
\end{lemma}
\begin{proof}
We define a \emph{filling rule} to transform a filling $F$
in $\CoInvFree(\alpha,\sigma)$ to some $F' \in \CoInvFree(s_1 \alpha,\sigma) \sqcup \CoInvFree(s_1 \alpha, s_1\sigma)$:
Start at the end of the first row of $F$ and map that entry to the bottom row (and corresponding column) in $F'$.
Label this entry $C$ and consider the adjacent column to the left.

If this adjacent column in $F$ has one entry, map that entry to the bottom row of $F'$ and let $C$ now denote this entry.

Otherwise, this column has entries $\{A,B\}$. If either entry is greater than $C$, choose the least 
element greater than $C$ to be in the bottom row. If neither entry is greater than $C$, 
choose the smallest entry to be in the bottom row. Repeat this procedure by letting $C$ be the bottom 
entry in $F'$ in the column to the right of the column being processed. 
It suffices to show that the resulting $F'$ is both coinversion-free and that it has the same major index as $F$.
\medskip 

The first part is easy --- note that $A \neq B$ since otherwise, $F$ would contain a coinversion.
Then exactly one of the arrangements of $A$ and $B$ can produce a coinversion together with $C$,
and it is straightforward to verify that the filling rule gives the coinversion-free choice.

%

The main difficulty therefore is demonstrating that the filling rule is major index preserving.
We proceed by strong induction on the length of the shortest row $\alpha_2$
to demonstrate that the map is major index preserving.
For the remainder of the proof, let $<$ indicates the presence of a 
descent between two entries while $\ge$ indicates the lack of such a descent.
For the base case, suppose that the shorter row has length $0$.

Suppose $F$ has the form
\begin{align*}
&a_1~c_1~c_2~\cdots~c_m\\
&b_1
\end{align*} 
where $a_1$ and $b_1$ are the basement entries.
There are two separate cases to consider.

\noindent 
\textbf{Base case a: $c_1>a_1$.} Since $a_1>b_1$ it follows that $F'$ has the form
\begin{align*}
&a_1\\
&b_1 < c_1~c_2~\cdots~c_m
\end{align*}
and therefore both the initial and final fillings have a descent between the 
first and second columns and the major index is preserved. 

\noindent
\textbf{Base case b: $c_1\le a_1$.} Then $F'$ is of the form
\begin{align*}
\begin{matrix}
 a_1 \\
 b_1 & \geq & c_1 & c_2 & \dotsc & c_m
\end{matrix}
\qquad 
\text{ or }
\qquad 
\begin{matrix}
 b_1 \\
 a_1 & \geq & c_1 & c_2 & \dotsc & c_m
\end{matrix}
\end{align*}
and in either case, there is no descent between the first and second column.
Hence, the major index is preserved and we have proved the base case.
\bigskip

For all the general cases below, we use filling
\begin{align*}
F \quad = \quad
\begin{matrix}
 a_1 & a_2 & \dotsc & a_k & c_1 & \cdots & c_m \\
 b_1 & b_2 & \dotsc & b_k
\end{matrix}
\end{align*}
to represent an $F\in \CoInvFree(\alpha,\sigma)$
where $a_1$, $b_1$ is the basement with $a_1>b_1$. As before, $F'$ is the filling obtained from $F$
via the filling rule. We proceed by casework based on how the filling rule was applied.
\medskip 

\emph{The main outline for this casework is as follows:}
There are a few cases to consider depending on if there is some (smallest) index $\ell$, $1<\ell \leq k$
such that $a_\ell>b_\ell$, and if so, how this column appears in $F'$.
Note that we can use the induction hypothesis on all descents to the right of such a column,
and conclude that these contribute the same amount to the major index in the two fillings.
\medskip

\noindent
\textbf{Case 1:} There exist some $\ell \geq 3$, such that $a_\ell>b_\ell$,
$a_i<b_i$ for $2\le i\le \ell-1$ and $a_\ell$ appears on top of $b_\ell$ in $F'$.
Remember, $F$ is of the form 
\begin{align*}
&a_1~\cdots~a_\ell~\cdots~a_k ~c_1~c_2~\cdots~c_m\\
&b_1~\cdots~b_\ell~\cdots~b_k 
\end{align*}
By the inductive hypothesis, we have that the major index is preserved after the $\ell^{th}$ column.\footnote{Note that we are essentially treating the $\ell^{th}$ column as a basement in order to invoke the inductive hypothesis. This is done throughout the proof when invoking the inductive hypothesis.}

\noindent
\textbf{Subcase 1a:} Suppose $b_{\ell-1}<b_\ell$.
It follows that $a_{\ell-1}<b_{\ell-1}<b_{\ell}<a_{\ell}$. 
Since $F$ has no coinversions it follows 
that $b_{i+1}>a_i\ge a_{i+1}$ for $2\le i\le \ell-2$ and since $b_{i}>a_{i}$ over the same 
indices it follows that $b_{i}>a_i\ge a_{i+1}$. 
Therefore --- using the filling rule --- it follows that $F'$ is of the form
\begin{align*}
&\sigma(a_1,b_1)~b_2~\cdots~b_{\ell-1}< a_\ell~\cdots\\
&\sigma(a_1,b_1)~a_2~\cdots~a_{\ell-1}< b_\ell~\cdots~c_1~\cdots~c_m
\end{align*}
and the major index is clearly preserved from the second column onwards. 
To prove that the entries in the first and second column preserves major index,
there are two subsubcases to consider. If $b_1\ge a_2$, then $b_2>a_1>b_1\ge a_2$ since $F$ 
has no coinversions. Thus, 
\begin{align*}
F = 
\begin{matrix}
 a_1 & \geq & a_2 & \dotsc \\
 b_1 & < & b_2 & \dotsc
\end{matrix}
\qquad
\text{and}
\qquad
F' = 
\begin{matrix}
 a_1 & < & b_2 & \dotsc \\
 b_1 & \geq & a_2 & \dotsc
\end{matrix}
\end{align*}
Since the descent in both fillings is in the shortest row, it 
is clear that in both cases it has the same contributes to the major index.
\medskip

In the second subsubcase, we have $a_2>b_1$, and since $F$ has no coinversions this implies that $b_2>a_1\ge a_2>b_1$. 
Now we have
\begin{align*}
F = 
\begin{matrix}
 a_1 & \geq & a_2 & \dotsc \\
 b_1 & < & b_2 & \dotsc
\end{matrix}
\qquad
\text{and}
\qquad
F' = 
\begin{matrix}
 b_1 & < & b_2 & \dotsc \\
 a_1 & \geq & a_2 & \dotsc
\end{matrix}
\end{align*}
and major index is preserved again.

\noindent
\textbf{Subcase 1b:} Otherwise suppose $b_{\ell-1}\ge b_{\ell}$. 

\noindent
\textbf{Subcase 1b.i:} If $a_{\ell-1}\ge b_{\ell}$ it follows 
that $b_{\ell-1}>a_{\ell-1}\ge b_{\ell}$. Since $F$ is coinversion-free, 
it follows that $b_{\ell}<a_{\ell}\le a_{\ell-1}<b_{\ell-1}$. 
Since $F$ is coinversion-free it also implies that $b_{i+1}>a_i\ge a_{i+1}$ for $2\le i\le \ell-2$.
Furthermore, since $b_i>a_i$ over the same indices, 
it follows that $b_i>a_i\ge a_{i+1}$ and $F'$ is of the form
\begin{align*}
&\sigma(a_1,b_1)~b_2~\cdots~b_{\ell-1}< a_\ell~\cdots\\
&\sigma(a_1,b_1)~a_2~\cdots~a_{\ell-1}< b_\ell~\cdots~c_1~c_2~\cdots~c_m.
\end{align*}
Using the same treatment as in Subcase 1a, it follows that $\maj F = \maj F'$.

\noindent
\textbf{Subcase 1b.ii:} Otherwise, $a_{\ell}>b_{\ell}>a_{\ell-1}$ and $b_{\ell-1}\ge b_{\ell}$. 

\begin{itemize}
 \item 
If $b_{\ell-1}>b_{\ell-2}$, it follows that $b_{\ell-1}>b_{\ell-2}>a_{\ell-2}\ge a_{\ell-1}$
and $F$ has the form
\begin{align*}
&\cdots~a_{\ell-2}\ge a_{\ell-1}< a_\ell~\cdots~c_1~c_2~\cdots~c_m\\
&\cdots~b_{\ell-2}< b_{\ell-1}\ge b_\ell~\cdots
\end{align*}
while $F'$ is of the form
\begin{align*}
&\cdots~b_{\ell-2}\ge a_{\ell-1}< a_\ell~\cdots\\
&\cdots~a_{\ell-2}< b_{\ell-1}\ge b_\ell~\cdots~c_1~c_2~\cdots~c_m.
\end{align*}
The major index is preserved before the $(\ell-2)^{nd}$ column by the reasoning in Subcase 1a.
Hence, the major index is preserved on the entire filling and statement follows.

\item 
Otherwise, $b_{\ell-1}\le b_{\ell-2}$. 
Since $F$ is coinversion-free, it follows that $a_{\ell-1}\le a_{\ell-2}<b_{\ell-1}$.
In this case, both $F$ and $F'$ have two rows which are of the form
\begin{align*}
&\cdots~a_{\ell-2}\ge a_{\ell-1}< a_\ell~\cdots\\
&\cdots~b_{\ell-2}\ge b_{\ell-1}\ge b_\ell~\cdots
\end{align*}
Continuing left in such a manner, we either reach an index 
such that $b_{j-1}<b_j$ for $3\le j\le {\ell-1}$ or there is no such index. 
If $b_{j-1}<b_j$ is the greatest such $j$, it follows using logic identical 
to above that $F$ has the form 
\begin{align*}
&\cdots a_{j-1}\ge a_j\ge a_{j+1}~\cdots~\ge a_{\ell-1}< a_\ell~\cdots~c_1~c_2~\cdots~c_m\\
&\cdots b_{j-1} < b_j\ge b_{j+1}~\cdots~\ge b_{\ell-1}\ge b_\ell~\cdots
\end{align*}
while $F'$ has the form 
\begin{align*}
&\cdots b_{j-1}\ge a_j\ge a_{j+1}~\cdots~\ge a_{\ell-1}< a_\ell~\cdots\\
&\cdots a_{j-1} <  b_j\ge b_{j+1}~\cdots~\ge b_{\ell-1}\ge b_\ell~\cdots~c_1~c_2~\cdots~c_m
\end{align*}
and the major index is preserved from column $j$ through column $\ell$. 
The remaining columns have major index preserved due to Subcase 1a, and thus major index is preserved overall. 
If there is no such $j$, note that $b_2>a_1\ge a_2$ since $F$ is coinversion-free
with $a_1>b_1$. Therefore, $F$ is of the form
\begin{align*}
& a_1\ge a_{2}~\cdots~\ge a_{\ell-1}< a_\ell~\cdots~c_1~c_2~\cdots~c_m\\
&b_1< b_{2}~\cdots~\ge b_{\ell-1}\ge b_\ell~\cdots
\end{align*}
while $F'$ is of the form
\begin{align*}
& a_1\ge a_{2}~\cdots~\ge a_{\ell-1}< a_\ell~\cdots\\
&b_1< b_{2}~\cdots~\ge b_{\ell-1}\ge b_\ell~\cdots~c_1~c_2~\cdots~c_m
\end{align*}
and major index is preserved as the sum of leg lengths are the same.
\end{itemize}

\bigskip 

\noindent
\textbf{Case 2:} There exist some $\ell \geq 3$, such that $a_\ell>b_\ell$,
$a_i<b_i$ for $2\le i\le \ell-1$ and $a_\ell$ appears below $b_\ell$ in $F'$.

\noindent
\textbf{Subcase 2a:} If $a_{\ell-1}\ge a_\ell$ then it follows that $a_{\ell-1}\ge a_{\ell}>b_\ell$. 
Therefore, $F'$ is of the form
\begin{align*}
&\cdots b_{\ell-1}~b_\ell~\cdots~a_k\\
&\cdots a_{\ell-1}~a_\ell~\cdots~b_k~c_1~c_2~\cdots~c_m
\end{align*}
and the major index is preserved using Subcase 1a. 

\noindent
\textbf{Subcase 2b:} Otherwise $a_{\ell}> a_{\ell-1}$.

\noindent
\textbf{Subcase 2b.i:} Suppose that $a_{\ell}>b_{\ell-1}$. 
Then note that $a_{\ell}>b_{\ell-1}>a_{\ell-1}$ and $F$ of the form 
\begin{align*}
&\cdots a_{\ell-1}~a_\ell~\cdots~a_k~c_1~c_2~\cdots~c_m\\
&\cdots b_{\ell-1}~b_\ell~\cdots~b_k
\end{align*}
while $F'$ is of the form 
\begin{align*}
&\cdots b_{\ell-1}~b_\ell~\cdots~a_k\\
&\cdots a_{\ell-1}~a_\ell~\cdots~b_k~c_1~c_2~\cdots~c_m
\end{align*}
and thus the major index is preserved using the logic of Subcase 1a.

\noindent
\textbf{Subcase 2b.ii:} Otherwise, it follows that $b_{\ell-1}\ge a_{\ell}>b_{\ell}$. 
Since $F$ is coinversion-free, it follows that $a_{\ell-1}<b_{\ell}<a_{\ell}$ 
and $F$ is of the form
\begin{align*}
&\cdots a_{\ell-1}< a_\ell~\cdots~a_k~c_1~c_2~\cdots~c_m\\
&\cdots b_{\ell-1}\ge b_\ell~\cdots~b_k
\end{align*}
which yields an $F'$ of the form 
\begin{align*}
&\cdots a_{\ell-1}< b_\ell~\cdots~a_k\\
&\cdots b_{\ell-1}\ge a_\ell~\cdots~b_k~c_1~c_2~\cdots~c_m.
\end{align*}
The major index is preserved in this case by following Subcase 1b.ii.

\noindent
\textbf{Case 3:} Case $\ell=2$. This means $a_2>b_2$.
By inductive hypothesis, the major index is preserved from the second column onward.
We only need to verify that major index is preserved among descents between first and second column.

\noindent
\textbf{Subcase 3a:} Suppose that $b_2>a_1$ then it follows that $a_2>b_2>a_1>b_1$. 
In this case, $F'$ is of the form
\begin{align*}
\begin{matrix}
 a_1 & < & a_2 & \dotsc \\
 b_1 & < & b_2 & \dotsc 
\end{matrix}
\qquad
\text{ or }
\qquad
\begin{matrix}
 a_1 & < & b_2 & \dotsc \\
 b_1 & < & a_2 & \dotsc 
\end{matrix}
\end{align*}
and
\begin{align*}
F\quad = \quad 
\begin{matrix}
 a_1 & < & a_2 & \dotsc \\
 b_1 & < & b_2 & \dotsc 
\end{matrix}
\end{align*}
so the major index is preserved, since the descents between first and second columns appear in both fillings.

\noindent
\textbf{Subcase 3b:} Otherwise $a_1\ge b_2$ and then there are two possible cases.

\noindent
\textbf{Subcase 3b.i:} $a_2$ appear below $b_2$ in $F'$:
\begin{itemize}
\item Case $b_1\ge a_2$: It follows that we have the indicated non-descents:
\begin{align*}
F=
\begin{matrix}
 a_1 & \geq & a_2 & \dotsc  \\
 b_1 & \geq & b_2 & \dotsc  
\end{matrix}
\quad
\text{ and }
\quad
F' = 
\begin{matrix}
 a_1 & \geq & b_2 & \dotsc  \\
 b_1 & \geq & a_2 & \dotsc 
\end{matrix}
\end{align*}
and thus the major index is preserved.
 
\item Case $b_1<a_2$: With $F$ being coinversion-free,  we have $a_1\ge a_2>b_2$, so
\begin{align*}
F=
\begin{matrix}
 a_1 & \geq & a_2 & \dotsc \\
 b_1 & \ast & b_2 & \dotsc 
\end{matrix}
\quad
\text{ and }
\quad
F' = 
\begin{matrix}
 b_1 & \ast & b_2 & \dotsc \\
 a_1 & \geq & a_2 & \dotsc 
\end{matrix}
\end{align*}
where $\ast$ contributes in the same way to both fillings if it is a descent.
\end{itemize}

\noindent
\textbf{Subcase 3b.ii:} $a_2$ appear above $b_2$ in $F'$: 
There are again two possible cases.
\begin{itemize}
 \item Case $b_1\ge b_2$: This condition together with $F$ being coinversion-free implies that 
$b_2<a_2\le a_1$ and we must have
\begin{align*}
F=
\begin{matrix}
 a_1 & \geq & a_2 & \dotsc  \\
 b_1 & \geq & b_2 & \dotsc
\end{matrix}
\quad
\text{ and }
\quad
F' = 
\begin{matrix}
 a_1 & \geq & a_2 & \dotsc  \\
 b_1 & \geq & b_2 & \dotsc 
\end{matrix}
\end{align*}
\item Case $b_1<b_2$: This gives $a_1\ge a_2>b_2>b_1$ and 
it follows that
\begin{align*}
F=
\begin{matrix}
 a_1 & \geq & a_2 & \dotsc  \\
 b_1 &  < & b_2 & \dotsc
\end{matrix}
\quad
\text{ and }
\quad
F' = 
\begin{matrix}
 a_1 &  < & a_2 & \dotsc  \\
 b_1 & \geq & b_2 & \dotsc 
\end{matrix}
\end{align*}
\end{itemize}
In both these subcases, major index is preserved.

\noindent
\textbf{Case 4:} The last case is occurs whenever $a_1>b_1$ and $a_i<b_i$ for $2\le i\le k$.

\noindent
\textbf{Subcase 4a:} If $c_1\le a_{k}<b_{k}$ or $a_{k}<b_{k}\le c_1$ then $F'$ is
\begin{align*}
&\cdots~b_2~\cdots b_{k}\\
&\cdots~a_2~\cdots a_{k}~c_1~c_2~\cdots~c_m\\
\end{align*}
and using the logic in Subcase 1a the result follows.

\noindent
\textbf{Subcase 4b:} Otherwise $a_{k}<c_1\le b_{k}$. Then we have
\begin{align*}
F\quad=\quad
\begin{matrix}
 \dotsc & a_k & < & c_1 & \dotsc \\
 \dotsc & b_k
\end{matrix}
\quad
\text{ and }
\quad
F'\quad=\quad
\begin{matrix}
 \dotsc &  a_k \\
 \dotsc &  b_k & \geq &  c_1 & \dotsc 
\end{matrix}
\end{align*}
Using the reasoning from Subcase 1b.ii, 
there exists a $j \geq 1$ such that 
\begin{align*}
F\quad=\quad
\begin{matrix}
 \dotsc & a_j & \geq & a_{j+1}& \geq & \dotsc & a_k & < &c_1 & \dotsc \\
 \dotsc & b_j &  < & b_{j+1}& \geq & \dotsc & b_k 
\end{matrix}
\end{align*}
and
\begin{align*}
F'\quad=\quad
\begin{matrix}
 \dotsc & a_j & \geq & a_{j+1}& \geq & \dotsc & a_k \\
 \dotsc & b_j &  < & b_{j+1}& \geq & \dotsc & b_k  & \geq &c_1 & \dotsc 
\end{matrix}
\end{align*}
The two marked descents in $F$ has leg lengths that sum to
the leg length of the single marked descent in $F'$.
Furthermore the major index is preserved before the $j^{th}$ using the logic of Subcase 1a and the therefore the major index is preserved overall.


%
\medskip 

All cases have now been covered, and this concludes the proof.
\end{proof}
The remainder of the proof is the verification that this two row filling rule is compatible 
with the remainder of the filling ---
that is, applying the filling rule on two rows in a larger filling,
no coinversions are introduced.

Since the presence of a coinversion is a local condition,
we can reduce the proof to a finite list of possible configurations.
We verify these via computer verification and the procedure
is described in the detail in the following proof.

\begin{lemma}\label{lem:fillingcompatible}
The filling rule in \cref{lem:tworowCase} produces 
fillings that are compatible with entries in a larger filling, \emph{i.e.},
no coinversions are introduced.
\end{lemma}
\begin{proof}
Since a coinversion only depends on elements in two separate rows,
it suffices to verify the algorithm for augmented fillings with 
precisely three rows --- the two rows affected by the filling rule, and a fixed third row.
We may then assume that these three rows are adjacent.
\medskip 

The general strategy is to consider a $2\times 3$ sub-grid of the diagram in $F$,
and then conclude that there are no coinversions in $F'$.
Furthermore, it suffices to restrict the entries between $1$ and $6$,
since only the relative order among the entries determines the presence of coinversions.
Let the two by three grid be as below
\begin{align*}
\begin{matrix}
a & d \\
b & e \\
c & f
\end{matrix}.
\end{align*}
Since inversions depend on the relative order of the row lengths,
we need to take that into consideration, by assigning each row a \emph{rank},
a number between $1$ and $3$, indicating the relative order of the row lengths.
In case two rows have the same length, the bottom row is assigned a lower rank.
The rank together with the up to six entries in the grid
allow us to determine which triples in the grid that are coinversions,
and we only consider such possible grids in $F$ which are 
coinversion-free.
\medskip

Some of the entries in the grid might not be present in the case we are examining 
the far right of a filling.
However, it is straightforward to see that there is no loss of generality 
to assume that $a$, $b$ and $c$ are always present, as the other cases
can easily be verified by hand.

The filling rule processes the two rows being swapped from right to left,
and we need to verify that there is no coinversion in $F'$ being produced involving 
the third fixed row.
Note that there might locally be \emph{two possible images}, $F'$ and $F''$
given a local grid $F$, depending on what the filling rule implied in the second column
--- in this example, the top two rows are swapped:
\begin{align}\label{eq:exampleSwapCases}
F =
\begin{matrix}
3: & a & d \\
1: & b & e \\
2: & c & f
\end{matrix}
\qquad
F' =
\begin{matrix}
1: & \ast & d \\
3: & \ast & e \\
2: & c & f
\end{matrix}
\qquad 
F'' =
\begin{matrix}
1: & \ast & e \\
3: & \ast & d \\
2: & c & f
\end{matrix}
\end{align}
The entries marked $\ast$ are permutations of $a$ and $b$, and the positions 
are determined by the filling rule. By construction, there are no coinversions involving only the rows being swapped,
so it suffices to check coinversions involving the third, fixed row.
\bigskip 

There are exactly three things that can occur locally in the grid.
We verify this using the computer.

\noindent
\textbf{The diagram is degenerate.}
One of the rows being swapped has an element missing in the grid.
If this is the case, the position of the grid entries in $F'$ are uniquely determined
by the filling rule and there is only one case in \eqref{eq:exampleSwapCases}.
By checking all such local cases, we see that all corresponding grids are coinversion-free.

\noindent
\textbf{Both possibilities are valid.}
\item The non-fixed entries in the second column can either be swapped or not by the filling rule,
and \emph{both} these possibilities yield a coinversion-free grid using the filling rule.
This is verified by computer. As an example of this situation, we might have
\[
F =
\begin{matrix}
3: & 4 & 2 \\
1: & 1 & 1 \\
2: & 6 & 5
\end{matrix}
\qquad
F' =
\begin{matrix}
1: & 4 & 2 \\
3: & 1 & 1 \\
2: & 6 & 5
\end{matrix}
\qquad 
F'' =
\begin{matrix}
1: & {1} & 1 \\
3: & {4} & {2} \\
2: & {6} & 5
\end{matrix}
\]
and in either case, the filling rule produces valid (coinversion-free) grid.

\noindent
\textbf{Only one of the grids is coinversion-free.}

This situation requires a more careful analysis,
and we need to do a non-local analysis to prove that $F'$ is indeed of the form
that produce a coinversion-free filling.

Computer check verifies that the event that only one of the two grids are valid
occurs only under the conditions in the following claim, which then determines that we are in the 
case that produces a valid grid:
\begin{claim}\label{clm:badCaseFix}
Suppose we swap the longest and shortest row in the $2\times 3$ grid, as in
\begin{align}\label{eq:23diagramcases}
\begin{matrix}
{\text{swap}\left.\rule{0pt}{12pt}\right\{ \; }\\
\phantom{a} 
\end{matrix}
\begin{matrix}
3: & a & d \\
1: & b & e \\
2: & c & f
\end{matrix}
\qquad 
\text{ or }
\qquad
\begin{matrix}
\phantom{a} \\
{\text{swap}\left.\rule{0pt}{12pt}\right\{ \; }
\end{matrix}
\begin{matrix}
2: & a & d \\
3: & b & e \\
1: & c & f
\end{matrix}
\end{align}
and one of $e>d>f$ or $d>f>e$ or $f>e>d$ hold (down-increasing condition).
Then the corresponding grid in $F'$ \emph{must} be of the respective forms
\begin{align}\label{eq:23diagramcases2}
\begin{matrix}
1: & \ast & e \\
3: & \ast & d \\
2: & c & f
\end{matrix}
\qquad 
\text{ and }
\qquad
\begin{matrix}
2: & a & d \\
1: & \ast & f \\
3: & \ast & e
\end{matrix}
\end{align}
\end{claim}
\begin{proof}
Suppose that the entries in the second column of $F'$ are not as in \cref{eq:23diagramcases2},
that is, we assume they did not ``flip''.

If the adjacent column to the right of the second column in $F$ also has all three entries present, 
it follows (via computer verification) that these entries
also have the down-increasing condition.
This third column in $F$ then also appears identically in $F'$.
The down-increasing condition is therefore an invariant, present in all further columns to the right,
via induction.

Eventually, we reach the end one of the shortest row, where last complete column satisfies
the down-increasing condition and is identical in both $F$ and $F'$.
Finally, an exhaustive search on the computer shows that 
it is impossible for $F$ and $F'$ to be of these specified forms and simultaneously be coinversion-free, a contradiction.

\medskip


Hence, the entries in $F'$ must be arranged as in \cref{eq:23diagramcases2}.
Computer verification on the local situation in \cref{eq:23diagramcases2} verifies that 
the filling rule produces no coinversions.
\end{proof}

To conclude the proof, here is a summary of properties that needs to be verified via computer:
\begin{itemize}
 \item Only the three local cases listed above appear among all possible $2\times 3$-grids;
 degenerate, two valid possibilities, one valid possibility of specified form.
 \item The filling rule in \cref{clm:badCaseFix} preserve the down-increasing condition.

 \item Preserving the down-increasing condition eventually leads to a contradiction --- that is,
  once we reach the end of the shortest row, $F$ and $F'$ cannot simultaneously be coinversion-free.
 This forces $F$ to be of the expected form.

\end{itemize}
We have done this in Mathematica and Java with two different implementations.
\end{proof}

\section{Another column-set preserving map}\label{sec:main2}

It is possible to mimic the above proposition in the case
of modified Hall--Littlewood polynomials --- fillings without inversions.

\begin{proposition}
Let $\InvFree(\alpha,\sigma)$ be the set of inversion-free fillings with 
shape $\alpha$ and big basement $\sigma$.
Suppose $\sigma_i >\sigma_{i+1}$ and $\alpha_i \ge \alpha_{i+1}$.
Then there is a bijection
\[
 \varphi : \InvFree( \alpha,\sigma) \longleftrightarrow \InvFree(s_i \alpha, s_i\sigma)
\]
with the property that $\varphi$ preserves column sets and major index. 
\end{proposition}
First note that the above proposition it trivial when $a_i=a_{i+1}$ as we can simply 
interchange the two basement entries and preserve the remainder of the filling, since this action
does not introduce inversions and it clearly preserves the major index.
In the case when $a_i>a_{i+1}$,
we proceed as before and first establish a bijection in the two row case,
followed by proving that this is compatible with a larger filling.

\medskip 

\begin{lemma}\label{lem:tworowCase2}
Let $\InvFree(\alpha,\sigma)$ be the set of inversion-free two-row fillings with 
shape $\alpha=(\alpha_1,\alpha_2)$ and big basement $\sigma=(\sigma_1,\sigma_2)$,
with $\sigma_1 >\sigma_{2}$ and $\alpha_1 > \alpha_{2}$.
Then there is a bijection $\varphi$
\[
 \varphi : \InvFree( \alpha,\sigma) \longleftrightarrow \InvFree(s_1 \alpha, s_1\sigma)
\]
with the property that $\varphi$ preserves column sets and major index.
\end{lemma}
\begin{proof}
We first consider the case when $\alpha_2=0$; in this case we can simply interchange
the rows and the result follows trivially. Otherwise for the remainder of the proof,
treat the column immediately after the basement to be the first column and note
the entries this column are in non-increasing order in order for the filling to have no inversion.

We now define a filling rule to transform a filling $F$ in $\InvFree(\alpha,\sigma)$
to a $F'\in \InvFree(s_1 \alpha,s_1\sigma)$: Start at the end of the first row
and map that entry into to the bottom row in $F'$.
Label this entry $C$ and consider the adjacent column to the left.

If this adjacent column in $F$ has one entry, map that entry to the bottom row
of $F'$ and let $C$ denote this entry.

Otherwise, the adjacent column in $F$ has entries $\{A,B\}$. If both entries are greater
than or equal to $C$ or less than $C$, map the larger entry to be in the bottom row
of $F'$. Otherwise map the lesser entry to be in the bottom row of $F'$.
Repeat this procedure with $C$ being the bottom entry in the adjacent column to
the right of the column being processed. It suffices to demonstrate that $F'$ is
inversion-free and that it has the same major index as $F$.
The first part is note that given a particular $\{A,B\}$ there is exactly one
arrangement of $A$ and $B$ that does not give a coinversion with $C$, and it
follows upon a straightforward verification that the
filling rule gives the inversion-free choice.

We now proceed by strong induction on the length of the shorter row $\alpha_2$ to
demonstrate that the filling rule is major index preserving. As in the previous
proof let $<$ indicate the presence of a descent between two entries while $\ge$
indicates the lack of a descent. For the base case, suppose that the shorter row
has length $1$. Suppose that $F$ has the form 
\begin{align*}
&a_1~c_1~c_2~\dotsc~c_m\\
&b_1
\end{align*}
where $a_1$ and $b_1$ are the first entries after the basement.
Then there are two separate cases to consider.

\noindent
\textbf{Base case a:} $a_1\geq  c_1$. Since $a_1\le b_1$ it follows that $F'$ has the form
\begin{align*}
&a_1\\
&b_1 \ge c_1~c_2~\dotsc~c_m\\
\end{align*}
and thus both $F$ and $F'$ lack a descent between the first and second columns and the major index is preserved.

\noindent
\textbf{Base case b:} $a_1< c_1$. The $F'$ is therefore
\begin{align*}
\begin{matrix}
 a_1 \\
 b_1 & < & c_1 & c_2 & \dotsc & c_m
\end{matrix}
\qquad 
\text{ or }
\qquad 
\begin{matrix}
 b_1 \\
 a_1 & < & c_1 & c_2 & \dotsc & c_m
\end{matrix}
\end{align*}
and in both cases there is a descent between the first and second column as in the original filling.
Hence the major index is preserved and the base case is proved.
\medskip

For all the cases below let original filling $F$ be represented by
\begin{align*}
F \quad = \quad
\begin{matrix}
 a_1 & a_2 & \dotsc & a_k & c_1 & \dotsc & c_m \\
 b_1 & b_2 & \dotsc & b_k
\end{matrix}
\end{align*}
were $a_1$ and $b_1$ are the first element \emph{after} the basement and $a_1\le b_1$.
As before, $F'$ is the filling obtained from the reverse filling rule.
We now proceed by casework on how the filling appears after the filling rule is used.

\noindent
\textbf{Case 1:} There exists $\ell\ge 3$ such that $a_\ell\le b_\ell$, $a_i>b_i$
for $2\le i\le \ell-1$ and $a_\ell$ appears on top on $b_\ell$ in $F'$.
By inductive hypothesis, the major index is preserved after the $\ell^{th}$ column.\footnote{This can be done as taking the original basement
and appending the filling as it appears at the $\ell^{th}$ column and after gives an inversion free filling. This logic is used whenever the inductive hypothesis is invoked.}

\noindent
\textbf{Subcase 1a:} Suppose that $b_{\ell-1}\ge b_\ell$.
It follows that $a_{\ell-1}>b_{\ell-1}\ge b_{\ell}\ge a_{\ell}$.
Since $F$ is inversion-free, it follows that $b_{i+1}\le a_i< a_{i+1}$ for $2\le i\le \ell-2$
and since $a_i>b_i$ over the same indices it follows that $b_i<a_i<a_{i+1}$.
Therefore, using the filling rule it follows that $F'$ is
\begin{align*}
&\sigma(a_1,b_1)~b_2~\dotsc~b_{\ell-1}\ge a_\ell~\dotsc\\
&\sigma(a_1,b_1)~a_2~\dotsc~a_{\ell-1}\ge b_\ell~\dotsc~c_1~c_2~\dotsc~c_m
\end{align*}
and the major index is clearly preserved from the second column onwards.
To prove that the major index is preserved between the first and second column,
there are two cases. If $b_1< a_2$, then $b_2\le a_1\le b_1< a_2$ as $F$ has no inversions.
Thus
\begin{align*}
F = 
\begin{matrix}
 a_1 & < & a_2 & \dotsc \\
 b_1 & \geq & b_2 & \dotsc
\end{matrix}
\qquad
\text{and}
\qquad
F' = 
\begin{matrix}
 a_1 & \geq & b_2 & \dotsc \\
 b_1 & < & a_2 & \dotsc
\end{matrix}
\end{align*}
Since the descent is in the longest row in both fillings, it is clear that in both cases it has the same contribution to major index. 

In the second subcase, $a_2\le b_1$ and since $F$ has no inversions
it follows that $b_2\le a_1< a_2\le b_1$.
In this case it follows that we are in the situation
\begin{align*}
F =
\begin{matrix}
 a_1 & < & a_2 & \dotsc \\
 b_1 & \geq & b_2 & \dotsc
\end{matrix}
\qquad
\text{and}
\qquad
F' =
\begin{matrix}
 b_1 & \geq & b_2 & \dotsc \\
 a_1 & < & a_2 & \dotsc
\end{matrix}
\end{align*}
and major index is preserved in this case as well.

\noindent
\textbf{Subcase 1b:} Otherwise $b_{\ell-1}< b_{\ell}$. 

\noindent
\textbf{Subcase 1b.i:} If $a_{\ell-1}< b_{\ell}$ then it follows $b_{\ell-1}<a_{\ell-1}<b_{\ell}$.
Furthermore since $F$ has no inversions it follows that $b_{\ell-1}<a_{\ell-1}<a_{\ell}\le b_{\ell}$.
This property also implies that $b_i<a_i<a_{i+1}$ for $2\le i\le \ell-2$ and thus $F'$ has the form
\begin{align*}
&\dotsc~b_{\ell-1}< a_\ell~\dotsc\\
&\dotsc~a_{\ell-1}< b_\ell~\dotsc~c_1~c_2~\dotsc~c_m.
\end{align*}
Using the same treatment in Subcase 1a, it follows that major index is preserved.

\noindent
\textbf{Subcase 1b.ii:} Otherwise $a_{\ell}\le b_{\ell}\le a_{\ell-1}$ and $b_{\ell-1}<b_{\ell}$. 
\begin{itemize}
\item 
If $b_{\ell-1}\le b_{\ell-2}$ then it follows that $b_{\ell-1}\le b_{\ell-2}<a_{\ell-2}< a_{\ell-1}$ as
there are no inversions in $F$. Then $F$ and $F'$ are of the forms
\begin{align*}
F =
\begin{matrix}
\dotsc  a_{\ell-2} & < & a_{\ell-1} & \geq & a_{\ell} \dotsc \\
\dotsc  b_{\ell-2} & \geq & b_{\ell-1} & < & b_{\ell} \dotsc
\end{matrix}
\end{align*}
and
\begin{align*}
F' =
\begin{matrix}
\dotsc  b_{\ell-2} & < & a_{\ell-1} & \geq & a_{\ell} \dotsc \\
\dotsc  a_{\ell-2} & \geq & b_{\ell-1} & < & b_{\ell} \dotsc
\end{matrix}
\end{align*}
and the major index is preserved up to the $(\ell-2)^{nd}$ column by the reasoning in Subcase 1a. 
Therefore the major index is preserved on the entire filling and statement follows.
\item
Otherwise, $b_{\ell-1}> b_{\ell-2}$. Since $F$ is inversion-free it follows that $b_{\ell-1}\le a_{\ell-2}<a_{\ell-1}$. In this case, both $F$ and $F'$ have two rows which are of the form
\begin{align*}
\begin{matrix}
\dotsc  a_{\ell-2} & < & a_{\ell-1} & \geq & a_{\ell} \dotsc \\
\dotsc  b_{\ell-2} & < & b_{\ell-1} & < & b_{\ell} \dotsc
\end{matrix}.
\end{align*}
Continuing left, we either reach an index such that $b_{j+1}\le b_{j}$ for $2\le j\le {\ell-2}$ or there is no such index.
If $b_{j+1}\le b_j$ is the greatest such index it follows using identical reasoning as above that
\begin{align*}
\setcounter{MaxMatrixCols}{20}
F=
\begin{matrix}
\dotsc & a_j & <    & a_{j+1} & \dotsc &  < & a_{\ell-1} & \geq & a_{\ell} & \dotsc & c_1 & \dotsc & c_m\\
\dotsc & b_j & \geq & b_{j+1} & \dotsc &  < & b_{\ell-1} & < & b_{\ell} & \dotsc
\end{matrix}
\end{align*}
while $F'$ has the form
\begin{align*}
F'=
\begin{matrix}
\dotsc & b_j & <    & a_{j+1} & \dotsc &  < & a_{\ell-1} & \geq & a_{\ell} & \dotsc \\
\dotsc & a_j & \geq & b_{j+1} & \dotsc &  < & b_{\ell-1} & < & b_{\ell} & \dotsc & c_1 & \dotsc & c_m
\end{matrix}
\end{align*}
and the major index is preserved between from column $j$ and column $\ell$.
The remaining columns have major index preserved by the logic of Subcase 1a, and
thus the major index is preserved overall. If there is no such $j$,
note that $a_2>b_2$ and $b_1\ge a_1$ and since $F$ is inversion-free it
follows that $b_2\le a_1<a_2$. Therefore $F$ is of the form

\begin{align*}
\setcounter{MaxMatrixCols}{20}
F=
\begin{matrix}
a_1 & <    & a_{2} & \dotsc &  < & a_{\ell-1} & \geq & a_{\ell} & \dotsc & c_1 & \dotsc & c_m\\
b_1 & \geq & b_{2} & \dotsc &  \geq & b_{\ell-1} & < & b_{\ell} & \dotsc
\end{matrix}
\end{align*}
while $F'$ has the form
\begin{align*}
F'=
\begin{matrix}
a_1 & <    & a_{2} & \dotsc &  < & a_{\ell-1} & \geq & a_{\ell} & \dotsc \\
b_1 & \geq & b_{2} & \dotsc &  \geq & b_{\ell-1} & < & b_{\ell} & \dotsc & c_1 & \dotsc & c_m
\end{matrix}
\end{align*}
and major index is preserved as the sum of leg lengths are the same.
\end{itemize}

\noindent
\textbf{Case 2:} There exist $3\le \ell\le k$ such that $a_\ell\le b_\ell$, $a_i>b_i$ for $2\le i\le \ell-1$ and $a_{\ell}$ is below $b_{\ell}$ in $F'$.

\noindent
\textbf{Subcase 2a:} If $a_{\ell-1}< a_\ell$ then it follows that $b_{\ell-1}<a_{\ell-1}< a_{\ell}\le b_\ell$. 
Thus $F'$ has the form
\begin{align*}
&\dotsc b_{\ell-1}~b_\ell~\dotsc~a_k\\
&\dotsc a_{\ell-1}~a_\ell~\dotsc~b_k~c_1~c_2~\dotsc~c_m
\end{align*}
and the major index is preserved using Subcase 1a.

\noindent
\textbf{Subcase 2b:} Otherwise $a_{\ell}\le a_{\ell-1}$.

\noindent
\textbf{Subcase 2b.i:} Suppose $a_{\ell}\le b_{\ell-1}$.
Since $a_{\ell-1}>b_{\ell-1}$ it follows that $a_{\ell-1}>b_{\ell-1}\ge a_{\ell}$.
Therefore, it follows that
\begin{align*}
F=
\begin{matrix}
\dotsc   a_{\ell-1} & a_{\ell} & \dotsc & c_1  \dotsc \\
\dotsc   b_{\ell-1} & b_{\ell} & \dotsc
\end{matrix}
\quad
\text{ and }
\quad
F'=
\begin{matrix}
\dotsc   b_{\ell-1} & b_{\ell} & \dotsc \\
\dotsc   a_{\ell-1} & a_{\ell} & \dotsc & c_1  \dotsc
\end{matrix}
\end{align*}
and the major index is preserved using the reasoning in Subcase 1a.

\noindent
\textbf{Subcase 2b.ii:} Otherwise $b_{\ell-1}< a_{\ell}\le b_{\ell}$, $a_{\ell-1}\ge a_{\ell}$.
Furthermore $a_{\ell-1}\ge b_{\ell}\ge a_{\ell}$ as $F$ has no
inversions and it follows that the $F$ and $F'$ have the form
\begin{align*}
F=
\begin{matrix}
\dotsc   a_{\ell-1}&\geq & a_{\ell}  \dotsc  a_k & c_1  \dotsc \\
\dotsc   b_{\ell-1}&  < & b_{\ell}  \dotsc  b_k
\end{matrix}
\quad
\text{ and }
\quad
F'=
\begin{matrix}
\dotsc   a_{\ell-1}& \geq & b_{\ell}  \dotsc  a_k \\
\dotsc   b_{\ell-1}& < & a_{\ell}  \dotsc   b_k & c_1  \dotsc
\end{matrix}
\end{align*}
The major index is preserved using the reasoning in Subcase 1b.ii

\noindent
\textbf{Case 3:} Case $\ell=2$. In this case $a_2\le b_2$ and by inductive hypothesis,
the major index is preserved from the second column onward.
We need therefore simply to verify that the major index is preserved among descent between the first and second column.

\noindent
\textbf{Subcase 3a:} Suppose that $b_2\le a_1$ then it follows that $a_2\le b_2\le a_1\le b_1$. 
In this case, we have
\begin{align*}
\setcounter{MaxMatrixCols}{20}
F=
\begin{matrix}
a_1 & \geq & a_{2}  \dotsc   a_{k} & c_1  \dotsc  c_m\\
b_1 & \geq & b_{2}  \dotsc   b_{k}
\end{matrix}
\end{align*}
while $F'$ is in one of the forms
\begin{align*}
\begin{matrix}
a_1 & \geq & a_{2}  \dotsc   a_{k} \\
b_1 & \geq & b_{2}  \dotsc   b_{k} & c_1  \dotsc  c_m
\end{matrix}
\quad
\text{ or }
\quad
\begin{matrix}
a_1 & \geq & b_{2}  \dotsc   a_{k} \\
b_1 & \geq & a_{2}  \dotsc   b_{k}  & c_1  \dotsc  c_m\\
\end{matrix}
\end{align*}
The major index is preserved between the first and second columns since
descents are not present in either possible position.

\noindent
\textbf{Subcase 3b:} Otherwise $b_2>a_1$ and then there are two possible cases.

\noindent
\textbf{Subcase 3b.i:} $a_2$ appears below $b_2$ in $F'$:
\begin{itemize}
\item
Case $b_1< a_2$, where it follows that $a_1\le b_1<a_2\le b_2$ so 
\begin{align*}
F=
\begin{matrix}
a_1 & < & a_{2}  \dotsc  & c_1  \dotsc  c_m \\
b_1 & < & b_{2}  \dotsc
\end{matrix}
\end{align*}
and
\begin{align*}
\quad
F'=
\begin{matrix}
a_1 & < & b_{2}  \dotsc   \\
b_1 & < & a_{2}  \dotsc   & c_1  \dotsc  c_m \\
\end{matrix}
\end{align*}
and major index is preserved. 
\item 
Case  $b_1\ge a_2$. With $F$ being inversion-free and $b_1\ge a_2$, $a_2\le b_2$, $a_1<b_2$
it follows that $a_1<a_2\le b_2$. Therefore,
\begin{align*}
F=
\begin{matrix}
a_1 & a_{2}  \dotsc  & c_1  \dotsc  c_m \\
b_1 & b_{2}  \dotsc
\end{matrix}
\quad
\text{ and }
\quad
F'=
\begin{matrix}
b_1 & b_{2}  \dotsc   \\
a_1 & a_{2}  \dotsc   & c_1  \dotsc  c_m 
\end{matrix}
\end{align*}
and major index is preserved.
\end{itemize}

\noindent
\textbf{Subcase 3b.ii:} $a_2$ appears above $b_2$ in $F'$:
\begin{itemize}
\item
Case $b_1\ge b_2$. It follows that $b_1\ge b_2>a_1$ and $F$ has no
inversions it follows that $a_1<a_2\le b_2\le b_1$. 
Thus, the situation is
\begin{align*}
F=
\begin{matrix}
a_1 & <    & a_{2}  \dotsc  & c_1  \dotsc  c_m \\
b_1 & \geq & b_{2}  \dotsc
\end{matrix}
\quad
\text{ and }
\quad
F'=
\begin{matrix}
b_1 & \geq & a_{2}  \dotsc   \\
a_1 & <    & b_{2}  \dotsc   & c_1  \dotsc  c_m 
\end{matrix}
\end{align*}
and major index is preserved. 
\item
Case $b_1< b_2$. Since $b_2>a_1$ and $F$ has no inversions it follows that $a_1<a_2\le b_2$. 
Furthermore $a_1\le b_1$ and we have
\begin{align*}
F=
\begin{matrix}
a_1 & <    & a_{2}  \dotsc  & c_1  \dotsc  c_m \\
b_1 & < & b_{2}  \dotsc
\end{matrix}
\quad
\text{ and }
\quad
F'=
\begin{matrix}
a_1 & < & a_{2}  \dotsc   \\
b_1 & <    & b_{2}  \dotsc   & c_1  \dotsc  c_m 
\end{matrix}
\end{align*}
and major index is preserved.
\end{itemize}

\noindent
\textbf{Case 4:} The final case occurs if $a_1\leq b_1$ and $a_i>b_i$ for $2\leq i\leq k$.

\noindent
\textbf{Subcase 4a:} If $c_1> a_{k}>b_{k}$ or $a_{k}>b_{k}\ge c_1$ then $F'$ is 
\begin{align*}
&\dotsc~b_2~\dotsc b_{k}\\
&\dotsc~a_2~\dotsc a_{k}~c_1~c_2~\dotsc~c_m
\end{align*}
using the logic of Subcase 1a and the major index is preserved similarly to Subcase 1a. 

\noindent
\textbf{Subcase 4b:} Otherwise $a_{\ell}\ge c_1> b_{\ell}$. 
We have
\begin{align*}
F=
\begin{matrix}
\dotsc & a_{\ell} & \geq & c_1  \dotsc  c_m \\
\dotsc & b_{\ell}
\end{matrix}
\quad
\text{ and }
\quad
F'=
\begin{matrix}
\dotsc & a_{\ell} \\
\dotsc & b_{\ell} & > & c_1  \dotsc  c_m 
\end{matrix}
\end{align*}
Using the reasoning from Subcase 1b.ii it follows that there exists $j\ge 1$ such that
\begin{align*}
F=
\begin{matrix}
\dotsc a_j & <    &a_{j+1}\dotsc< & a_{\ell}& \geq & c_1 \dotsc  c_m \\
\dotsc b_j & \geq &b_{j+1}\dotsc< & b_{\ell}
\end{matrix}
\end{align*}
and
\begin{align*}
F'=
\begin{matrix}
\dotsc b_j & <    &a_{j+1}\dotsc < & a_{\ell} \\
\dotsc a_j & \geq &b_{j+1}\dotsc < & b_{\ell}& < & c_1  \dotsc  c_m 
\end{matrix}
\end{align*}
The marked descents in $F$ and $F'$ have leg lengths that sum to the same value so 
the major index is preserved after column $j$ while the major index is preserved 
from column $1$ to column $j$ using Subcase 1a. 

We have now verified all possible cases and the result that the filling rule is 
major index preserving follows via strong induction on the length of the shorter row.

To demonstrate that is is a bijection,
we refer to Theorem\footnote{Theorem 5.1.1 only treats unaugmented fillings. However, since we fix a basement 
on both sides that cannot introduce inversions, the result follows.} 5.1.1 in \cite{Haglund2005Macdonald}, with $t=0$.
It follows that the two sets of fillings $\varphi$ map between are equinumerous.
This together with the fact that $\varphi$ is injective implies that $\varphi$ is also bijective.
\end{proof}

Analogous to \cref{prop:colSetPresInj}, it remains to show that $\varphi$ is compatible
with a larger filling. We proceed as before, and reduce this to a finite set of 
verifications. The approach is similar to the case when we treated coinversion-free fillings,
and the main difference is the details in \cref{clm:badCaseFix2}.

\begin{lemma}\label{lem:fillingcompatible2}
The map $\varphi$ applied to two adjacent rows in a larger filling
does not introduce any inversions.
\end{lemma}
\begin{proof}
As before, it suffices to consider the three-row case, where $\varphi$
is applied to two adjacent rows and the third row is fixed.
Again, we consider a $2\times 3$-grid with entries in $1,\dotsc,6$,
and all three entries in the first column are present.
There are again exactly three things that can occur locally in the grid.
We verify this using the computer.

\noindent
\textbf{The diagram is degenerate.}
One of the rows being swapped has an element missing in the grid
--- the entries in $F'$ are uniquely determined.
By checking all such local cases, we see that all corresponding grids are inversion-free.

\noindent
\textbf{Both possibilities are valid.}
\item The non-fixed entries in the second column can either be swapped or not by the filling rule,
and \emph{both} these possibilities yield an inversion-free grid using the filling rule.
This is verified by computer. As an example of this situation, we might have
\[
F =
\begin{matrix}
3: & 3 & 4 \\
1: & 5 & 6 \\
2: & 1 & 2
\end{matrix}
\qquad
F' =
\begin{matrix}
1: & 3 & 4 \\
3: & 5 & 6 \\
2: & 1 & 2
\end{matrix}
\qquad 
F'' =
\begin{matrix}
1: & {5} & 6 \\
3: & {3} & {4} \\
2: & {1} & 2
\end{matrix}
\]
and in either case, the filling rule produces valid (inversion-free) grid.

\noindent
\textbf{Only one of the grids is inversion-free.}

This situation requires a more careful analysis,
and we need to do a non-local analysis to prove that $F'$ is indeed of the form
that produce a inversion-free filling.

Computer check verifies that the event that only one of the two grids are valid
occurs only under the conditions in the following claim, which then determines that we are in the 
case that produces a valid grid:
\begin{claim}\label{clm:badCaseFix2}
Suppose we swap the longest and shortest row in the $2\times 3$ grid, as in
\begin{align}
\begin{matrix}
{\text{swap}\left.\rule{0pt}{12pt}\right\{ \; }\\
\phantom{a} 
\end{matrix}
\begin{matrix}
3: & a & d \\
1: & b & e \\
2: & c & f
\end{matrix}
\qquad 
\text{ or }
\qquad
\begin{matrix}
\phantom{a} \\
{\text{swap}\left.\rule{0pt}{12pt}\right\{ \; }
\end{matrix}
\begin{matrix}
2: & a & d \\
3: & b & e \\
1: & c & f
\end{matrix}
\end{align}
Furthermore, suppose that one of ${e<d<f}$ or $d<f<e$ 
or $f<e<d$ or $d=f$ and ${e\neq f}$ hold ({up-increasing condition}).
Then the corresponding grid in $F'$ \emph{must} be of the respective forms
\begin{align}\label{eq:23diagramcases32}
\begin{matrix}
1: & \ast & e \\
3: & \ast & d \\
2: & c & f
\end{matrix}
\qquad 
\text{ and }
\qquad
\begin{matrix}
2: & a & d \\
1: & \ast & f \\
3: & \ast & e
\end{matrix}
\end{align}
\end{claim}
\begin{proof}
Suppose that the entries in the second column of $F'$ are not as in \cref{eq:23diagramcases32},
that is, we assume they did not ``flip''.

If the adjacent column to the right of the second column in $F$ also has all three entries present, 
it follows (via computer verification) that these entries
also have the up-increasing condition.
This third column in $F$ then also appears identically in $F'$.
The up-increasing condition is therefore an invariant, present in all further columns to the right,
via induction.

Eventually, we reach the end one of the shortest row, where last complete column satisfies
the up-increasing condition and is identical in both $F$ and $F'$.
Finally, an exhaustive search on the computer shows that 
it is impossible for $F$ and $F'$ to be of these specified forms and simultaneously be inversion-free.

\medskip

Hence, the entries in $F'$ must be arranged as in \cref{eq:23diagramcases32}.
Computer verification on the local situation in \cref{eq:23diagramcases32} verifies that 
the filling rule produces no inversions.
\end{proof}

Thus, $\varphi$ is an injection and together with Theorem 5.1.1 in \cite{Haglund2005Macdonald},
it is also a bijection. It is also possible to prove that this is a bijection by
a computer-style method similar to what is done above.
\end{proof}

As a final comment on this section, note that the set $\InvFree(\alpha,\sigma)$ (with big basement) is empty 
if there are $i<j$ such that $\alpha_i < \alpha_j$ and $\sigma_i < \sigma_j$.
This is due to the fact that the first entry in row $j$, 
and the two basement entries $\sigma_i$ and $\sigma_j$ form an inversion triple.

\section{Applications}\label{sec:applications}

Despite the relatively lengthy proofs given, the realization the fact that these
column set preserving maps exist allow for a multitude of short corollaries.

We begin with a technical result that has been partially proved in a variety of special cases.
\begin{proposition}\label{prop:atMostOneCoInvFree}
Given fixed column sets, there exists at most one coinversion-free 
filling with shape $\alpha$ and decreasing basement.
\end{proposition}
\begin{proof}
To prove that there is at most one coinversion free filling 
with shape $\alpha$ and decreasing basement, simply apply 
the column set preserving operator $\phi$ until the shape becomes a partition.

In the partition case, any list of column sets with sizes compatible with $\alpha$
admits a unique coinversion-free filling. This unique filling can be constructed from the column sets
via the following iterative process going column by column, and in each column, top to bottom:
Given an entry $e$ in column $i$, the adjacent entry in column $i+1$
is found by taking the largest unused entry in column set $i+1$ less or equal to $e$,
and if there is no such element, simply take the largest entry.
It is easy to verify that this gives a coinversion-free filling and given $\macdonaldE_{\lambda}(\xvec;1,0)=e_{\lambda'}(\xvec)$ the result follows.
\end{proof}

\medskip 

Given a coinversion-free filling $F$, we construct the \emph{biword} of $F$ as follows:
Let the top row be the non-basement entries of $F$ listed in increasing order
and the bottom row be the corresponding columns the entries belong to,
listed in increasing order in case of a tie in the first row, see \cref{eq:biworexample} for an example.
\begin{align}\label{eq:biworexample}
\begin{ytableau}
 \mathbf{4} &3&3&2&1&3\\
 \mathbf{3} &2&2&1\\
 \mathbf{2} &1&4\\
 \mathbf{1}
\end{ytableau}
\qquad \longrightarrow \qquad
\begin{pmatrix}
1&1&1&2&2&2&3&3&3&4 \\
4&3&1&3&2&1&5&2&1&2
\end{pmatrix}
\end{align}
We let $\cw(F)$ denote the lower row in this biword, the \emph{charge word} of $F$.
Finally, define the charge of $F$ as $\charge(\cw(F))$,
where $\charge(\cdot)$ is defined as in \emph{e.g.}, \cite{qtCatalanBook},
by decomposing $\cw(F)$ into \emph{standard subwords}, 
computing $\charge(w)\coloneqq \maj(\reverse(w^{-1}))$ of each such subword $w$ (interpreted as a permutation), followed by adding the results.
The standard subwords are extracted iteratively by finding the rightmost occurrence of the 
smallest element, then scanning right to left for the next smallest element, 
looping around it necessary, and then repeating this process
until one has found an occurrence of largest element to the word. 
These letters form the first subword
and this process is repeated on the remaining letters until there are no more letters left.

For example, the word $1322133241214$ has the subword decomposition
$w_1=3214$, $w_2=3241$, $w_3=321$, $w_4=12$, extracted as
\begin{align*}
 132 213 \bar{3} 241 \bar{2} \bar{1}\bar{4} &\to 3214\quad (1)\\
 132 21 \phantom{3}\bar{3} \bar{2}\bar{4}\bar{1}\phantom{2}\phantom{1}\phantom{4}  &\to 3241 \quad (1) \\
 1\bar{3}2 \bar{2}\bar{1} \phantom{3}\phantom{3} \phantom{2}\phantom{4}\phantom{1}\phantom{2}\phantom{1}\phantom{4}   &\to 321 \quad (0) \\	
 \bar{1}\phantom{3}\bar{2} \phantom{2}\phantom{1} \phantom{3}\phantom{3} \phantom{2}\phantom{4}\phantom{1}\phantom{2}\phantom{1}\phantom{4}  &\to 12 \quad (1)
\end{align*}
The charge for the permutations is displayed to the right, so the total charge of the word is $3$.
\medskip

\begin{theorem}\label{thm:composition-charge}
Let $F$ be a coinversion-free filling of shape $\alpha$ with the key basement $w_0$.
Then $\charge(\cw(F))=\maj(F)$.
\end{theorem}
\begin{proof}
Note that $\cw(F)$ is uniquely determined by the column sets of $F$ by construction. 
Therefore, after applying $\phi$ repeatedly, it suffices to prove the 
theorem when $F$ is of partition shape, and from hereon, $F$ is assumed to be of partition shape.
Under this assumption, $F$ can be constructed from the column sets
via the following iterative process given in \cref{prop:atMostOneCoInvFree}. 
\medskip

\textbf{Claim:} The first column (the column adjacent to the basement) of $F$ is strictly decreasing. 
Suppose the largest entry in the basement (and thus in the filling) is $n$. 
Note that the $i^{th}$ largest entry in the first column is at most $n+1-i$, 
since all entries in a fixed column are distinct. 
Furthermore, since the $i^{th}$ largest element in the basement is exactly $n+1-i$,
the previous construction rule implies that the first column is strictly decreasing.
\medskip

Now note that the subwords obtained in the word decomposition of $\cw(F)$ naturally correspond to rows of $F$:
The claim ensures that the largest entry in the first column is in the topmost row.
It follows that entry is also the first entry of the first subword of $\cw(F)$.
Now, by the iterative process on how to recover $F$ from its column sets,
it is straightforward to see that the first subword of $\cw(F)$ corresponds to the first row of $F$.
By using the fact that the first column is decreasing,
this argument can be repeated for the remaining rows to 
prove that subword $j$ of $\cw(F)$ corresponds to row $j$ of $F$.

Finally, since the inverse of the $j^{th}$ subword $w_j$ can be seen to have the same relative
order as the entries in row $j$, it follows that $\maj(\reverse(w_j^{-1}))$ 
is equal to the contribution of row $j$ to $\maj(F)$.
Hence, $\maj(F)$ is equal to $\charge(\cw(F))$.
\end{proof}

Note that \cref{thm:composition-charge} critically relies on the presence of a decreasing basement,
and the equality does not hold for non-decreasing basements.
\medskip

We are now prepared to give an analog of the famous cocharge formula 
for partition case of the non-symmetric Macdonald polynomials. 
Note that the following identity can also be proven using properties of 
LLT polynomials proved in \cite{Haglund2005Macdonald} along with the symmetric cocharge formula, 
however this proof is far more in the style of the proof of the original cocharge formula given 
in \cite{Haglund2005Macdonald}. This particular version also has the advantage 
that seems to it generalizes to a non-symmetric setting, see \cref{conj:positiveKeyExp} below.
\begin{theorem}
Let $\lambda$ be a partition. 
Then
\[
\macdonaldE_{\lambda}(\xvec;q,0)=\sum_{\mu}\schurS_{\mu}(\xvec)\sum_{P \in \SSYT(\mu',\lambda')}q^{\charge(P)}.
\]
\end{theorem}
\begin{proof}
The partition case of \cref{prop:atMostOneCoInvFree}
implies that every possible second row (with strictly decreasing blocks in the second row),
in a biword can be obtained from some coinversion-free filling of shape $\lambda$.

By performing RSK on such biwords, every possible pair $(P,Q)$ of SSYT's 
with the insertion tableau $P$ having shape $\mu$ 
and the recording tableau $Q$ having shape $\mu'$ appears.
The proof of this is identical to the one given for the traditional RSK in \cite{StanleyEC2}.
Since charge is Knuth-invariant, two coinversion-free fillings $F_1$ and $F_2$ with identical
insertion tableau $P$ also have the same charge, and
\cref{thm:composition-charge} implies that  $\maj(F_1)=\maj(F_2)=\charge(P)$.
It now follows that
\begin{align*}
\macdonaldE_{\lambda}(\xvec;q,0) &=\sum_{F\in \CoInvFree(\lambda,w_0)} q^{\maj(F)}\xvec^F 
=  \sum_{F\in \CoInvFree(\lambda,w_0)} q^{\charge(c(F))} \xvec^F \\
&\stackrel{RSK}{=} \sum_{P\in \SSYT(\mu',\lambda')} q^{\charge(P)} \sum_{Q\in \SSYT(\mu)} \xvec^{Q} \\
&=  \sum_{\mu} \schurS_{\mu}(\xvec) \sum_{P\in \SSYT(\mu',\lambda')} q^{\charge(P)}
\end{align*}
as desired.
\end{proof}

We believe that the above theorem carries over in a more general setting:
\begin{conjecture}\label{conj:positiveKeyExp}
Let $\alpha$ be a composition. Then
\[
\macdonaldE_\alpha(\xvec;q,0) = \sum_{ F \in \CoInvFree(\alpha,w_0) } q^{\maj F}\xvec^F 
= \sum_{P \in \SSYT(\mu',\lambda')} q^{\charge(P)} \key_{\gamma(P,\alpha)}(\xvec)
\]
where $\lambda = \sort(\alpha)$ and $\gamma(P,\alpha)$ is a 
composition\footnote{That rearranges to the conjugate shape of $P$.} determined by $\alpha$ and $P$.

In particular, if $S_Q$ is the subset of fillings in $\CoInvFree(\alpha,w_0)$ 
with recording tableau $Q \in \SSYT(\mu',\lambda')$,
then $\sum_{F \in S_Q} \xvec^F$ is a key polynomial.
\end{conjecture}
Note that a recent combinatorial proof of the key expansion of $\macdonaldE_\alpha(\xvec;q,0)$
using \emph{weak dual equivalence} was given in \cite{AssafKostka}.

\medskip

The following example illustrates \cref{conj:positiveKeyExp}.
\begin{example}
We have the following expansions in the key basis:
\begin{align*}
\macdonaldE_{320}(\xvec;q,0) &= \key_{320} + q \key_{311} +  q \key_{221} +  q^2 \key_{221} \\
\macdonaldE_{203}(\xvec;q,0) &= \key_{203} + q \key_{113} +  q \key_{212} +  q^2 \key_{221} \\
\macdonaldE_{023}(\xvec;q,0) &= \key_{023} + q \key_{113} +  q \key_{122} +  q^2 \key_{221} \\
\end{align*}
The corresponding recording tableaux are
\[
 \young(11,22,3),\quad
 \young(112,2,3),\quad
 \young(113,22),\quad
 \young(112,23).
\]
Observe that all coefficients are identical, and only the 
composition indexing the key polynomials differ by some permutation.
Furthermore, several different permutations of the same partition might be present,
as with $\key_{212}$ and $\key_{221}$ above.
This refines the Kostka--Foulkes polynomials,  as different terms 
in the Kostka--Foulkes polynomials might be associated to different key polynomials
in the non-symmetric setting.
\end{example}
The question of finding an appropriate generalization of charge that explains this phenomena
was raised by Lascoux in \cite[p. 267--268]{LascouxPolynomialsBook},
and thus \cref{thm:composition-charge} answers this question.

\subsection{Inversion-free fillings and Hall--Littlewood polynomials}

Similar to \cref{prop:atMostOneCoInvFree}, we have the following result:
\begin{proposition}\label{prop:atMostOneInvFree}
Given fixed column sets there exists at most one inversion-free 
filling with shape $\alpha$ and big basement $\sigma$.
\end{proposition}
\begin{proof}
Proof is similar to the one in \cref{prop:atMostOneCoInvFree}.
\end{proof}

A number of analogous results are also possible for the modified Hall--Littlewood polynomials. 
The following identity is given in \cite[Thm. 5.1.1]{HaglundNonSymmetricMacdonald2008}:
Let $\tau$ be a permutation and $\lambda$ be a partition. Then
\begin{align}\label{eq:generalMacdonaldHIdentity}
\macdonaldH^{\tau w_0}_{\tau\lambda}(\xvec;q,t) = \macdonaldH_\lambda(\xvec;q,t).
\end{align}
The proof of this is rather indirect, and J.~Haglund asked for a bijective proof of this identity. 
With our bijection $\varphi$, we can prove the $t=0$ case:
\begin{theorem}
For every $\tau\in \symS_n$ and partition $\lambda$, 
there is a column-set preserving bijection that establish the identity
\begin{align}\label{eq:macdonaldHIdentity}
 \macdonaldH^{\tau w_0}_{\tau\lambda}(\xvec;q,0) = \macdonaldH_\lambda(\xvec;q,0).
\end{align}
\end{theorem}
\begin{proof}
Consider the fillings contributing to $\macdonaldH^{\tau w_0}_{\tau\lambda}(\xvec;q,0)$.
We may repeatedly apply $\varphi$ until the resulting fillings are 
the ones contributing to $\macdonaldH_\lambda(\xvec;q,0)$.

\end{proof}
Although it might be tedious to carry out the bijections, remember that the resulting 
bijection is uniquely defined by the column-set preserving property.

We have not been able to prove the more general \cref{eq:generalMacdonaldHIdentity}, 
but computer experiments suggests the following refinement of the equality:
\begin{conjecture}
Let $\FIL(\lambda,\tau,C)$ be all fillings with shape $\lambda$, big basement $\tau$
and column sets $C$.
Then
\[
 \sum_{F \in \FIL(\lambda,w_0,C)} q^{\maj(F)} t^{\inv(F)}
 =
 \sum_{F \in \FIL(\tau\lambda,\tau w_0,C)} q^{\maj(F)} t^{\inv(F)}.
\]
\end{conjecture}
This indicates that one should be able to find a \emph{column-set preserving bijection}
proving Haglund's identity.

\begin{remark}
One can modify \cref{lem:tworowCase} to show that for every $\tau$, we have
\[
 [t^{top}]\macdonaldH^{\tau w_0}_{\tau\lambda}(\xvec;q,t) = [t^{top}]\macdonaldH_\lambda(\xvec;q,t),
\]
where $[t^{top}]$ indicate the coefficient of the maximal power of $q$,
and the bijection proving this is given by $\phi$.
To give an outline of the proof, note first that fillings contributing to
both sides of the above identity are exactly coinversion-free fillings.
Furthermore, notice in this case with the big basement, is that the first set in the right hand side of
\cref{prop:colSetPresInj} is empty.
\end{remark}

\bigskip

We conclude this paper by proving an analogue of \cref{thm:composition-charge},
in the case of inversion-free fillings.
Let $F$ be an inversion-free filling with a basement.
Define the following biword of $F$ (different from the above) as follows:
Let the top row be the non-basement entries of $F$ listed in decreasing order
and the bottom row be the corresponding columns the entries belong to,
listed in \emph{increasing} order in case of a tie in the first row.
Let the cocharge word of $F$, $\ccw(F)$, be the bottom row in this biword.
As an example,
\[
\young(71,91423,8213,6)
\qquad \longrightarrow \qquad
\begin{pmatrix}
9 & 8 & 7 & 6 & 4 & 3 & 3 & 2 & 2 & 1 & 1 & 1 \\
1 & 1 & 1 & 1 & 3 & 4 & 5 & 2 & 4 & 2 & 2 & 3 \\
\end{pmatrix}
\]
has subword decomposition $13542$, $1423$, $12$, $1$,
and cocharge value $4+1+0+0=5$. 
\begin{theorem}
\label{TheoremCocharge}
Let $F$ be an inversion-free filling of composition 
shape $\tau \lambda$ where $\lambda$ is partition and basement $\tau w_0$.
Then 
\[
\cocharge(\ccw(F))=\maj(F).
\]
where $\cocharge(\cdot)$ is defined as in \cite{Haglund2005Macdonald}.
\end{theorem}
\begin{proof}
The case where $\alpha$ is partition is given in \cite{Haglund2005Macdonald}.
\footnote{The proposition in \cite{Haglund2005Macdonald} does not include a basement 
explicitly, but adding $\omega_0$ as basement in the partition case 
leaves the analysis in \cite{Haglund2005Macdonald} unchanged.}
The result then follows by noting that the $\ccw(F)$ is only dependent on the columns
sets of $F$, so by applying the column-set and major-index preserving map $\varphi$ 
until the partition shape is reached, the statement follows.
\end{proof}
This answers a conjecture given in \cite{Nelsen2005} \footnote{In \cite{Nelsen2005}, only fillings of 
shape $(v_1,v_2,\ldots,v_k)$ and an index $\ell$ such that $v_1\ge v_2\ge\ldots\ge v_{\ell}$ 
and $v_1<v_{\ell+1}<v_{\ell+2}\ldots<v_k$ are considered and the (implicit) basement is of the 
same form as in \cref{TheoremCocharge}.}

\subsection*{Acknowledgement}

The authors would like to thank Jim Haglund for helpful discussions. 
The second author would also like to thank the first author for his mentorship.
PA is funded by the \emph{Knut and Alice Wallenberg Foundation} (2013.03.07).

\bibliographystyle{amsalpha}
\bibliography{bibliography}

\end{document}